\documentclass[a4paper,10pt]{amsart}

\usepackage{graphicx}

\usepackage{cite}
\usepackage{todonotes}
\usepackage{hyperref}

\hypersetup{%
pdftitle={},
pdfsubject={Mathematics},
pdfauthor={Manfred Madritsch, Robert Tichy},
pdfkeywords={}
hyperindex=true,plainpages=false}

\usepackage{a4wide}

\usepackage{amsmath}
\usepackage{amsfonts}
\usepackage{amssymb}
\usepackage{amsthm}

\usepackage[initials]{amsrefs}

\usepackage{mathtools}

\usepackage{stmaryrd}

\newtheorem{lem}{Lemma}[section]
\newtheorem{thm}[lem]{Theorem}
\newtheorem{prop}[lem]{Proposition}

\numberwithin{equation}{section}

\newtheorem*{cor*}{Corollary}
\newtheorem*{thm*}{Theorem}

\theoremstyle{definition}
\newtheorem{defi}{Definition}[section]

\theoremstyle{remark}
\newtheorem{rem}[lem]{Remark}

\allowdisplaybreaks[3]

\newcommand{\sums}{\sideset{}{'}\sum}
\newcommand{\NN}{\mathbb{N}}
\newcommand{\ZZ}{\mathbb{Z}}
\newcommand{\QQ}{\mathbb{Q}}
\newcommand{\RR}{\mathbb{R}}
\newcommand{\CC}{\mathbb{C}}

\newcommand{\abs}[1]{\left| #1 \right|}
\newcommand{\norm}[1]{\left\| #1 \right\|}

\title[]{Distribution properties of generalized polynomials}

\dedicatory{To the memory of Vera Turán Sós}

\author[M. G. Madritsch]{Manfred G. Madritsch}
\address[M. G. Madritsch]{%
\noindent Technical University of Leoben, Chair for Applied Mathematics,
  8700 Leoben, Austria}
\email{manfred.madritsch@unileoben.ac.at}

\author[R. F. Tichy]{Robert F. Tichy}
\address[R. F. Tichy]{%
\noindent  Graz University of Technology, Institute of Analysis and Number
  Theory, 8010 Graz, Austria}
\email{tichy@tugraz.at}



\date{\today}

\begin{document}

\begin{abstract}
A generalized polynomial is a function defined by an iteration of the operations
addition, multiplication and the floor function. Equidistribution results of
sequences given by generalized polynomials have been established by H\aa land
and later by Bergelson and Leibman from an ergodic theoretic point of view. In
the present paper we show an asymptotic distribution result which can be applied
to certain generalized polynomials. In the second part we prove equidistribution
results for generalized polynomials along prime numbers, including bounds for
the discrepancy.
\end{abstract}

\maketitle

\section{Introduction}

Among the most famous results of V. T. Sós \cite{sos1958:distribution_mod_1} in
number theory is her proof of the so-called \textit{three-gap theorem} (also
referred to as \textit{Steinhaus conjecture} or \textit{three-distance
theorem}). This result is concerned with the Kronecker sequence
$(u_n)=(n\alpha)_{n=1}^\infty$ modulo $1$ where $\alpha$ is a real number. The
three-gap theorem asserts that for any real $\alpha$ there are at most three
different distances between neighboring elements of the sequence
$(n\alpha)_{n=1}^N$ modulo $1$. More precisely, let $S(\alpha)=\{n\alpha\colon
n=1,\ldots,N\}$ and $G(\alpha)$ be the set of
distances between neighboring elements, \textit{i.e.}
\[
  G(\alpha)=\left\{\min_{m=1,\ldots,N} \norm{n\theta-m\theta}\colon n=1,\ldots,N\right\},
\]
where $\norm{x}=\min_{n\in\ZZ}\abs{x-n}$ is the distance to the nearest
integer. Then the three-gap theorem states that $\left|G(\alpha)\right|\leq 3$.
For various proofs and comments see for instance Marklof and Strömbergsson
\cite{marklof_stroembergsson2017:three_gap_theorem} or Allouche and Shallit
\cite{allouche_shallit2003:automatic_sequences}.

Equidistribution properties of this sequence and of various subsequences have
been studied extensively. Let $f(x)$ be a continuous and strictly increasing
function $f: [0,1]\to [0,1]$ with $f(0)=0$ and $f(1)=1$. Then a sequence
$(x_n)_{n=1}^\infty$ of real numbers is called asymptotically distributed modulo
$1$ with distribution function $f$ if
\begin{gather}\label{eq:1}
  \lim_{N\to \infty} \frac1N\#\left\{n\leq N\colon  \left\{ x_n\right\}<x\right\}= f(x)
\end{gather}
for $0\leq x\leq 1$, where $\{x\}:=x-\lfloor x\rfloor$ denotes the fractional
part of the real number $x$. Introducing the discrepancy of $(x_n)$ with respect
to the distribution function $f(x)$
\begin{gather}\label{eq:2}
  D_N^{(f)}(x_n)=\sup_{0\le x<1} \left|\frac1N\#\left\{n\leq N\colon  \left\{ x_n\right\}<x\right\}-f(x)\right|,
\end{gather}
property \eqref{eq:1} is equivalent to
\begin{gather}\label{eq:3}
  \lim_{N\to \infty}D_N^{(f)}(x_n)= 0.
\end{gather}
Thus $ D_N^{(f)}(x_n)$ can be seen as a quantitative measure for asymptotic
distribution, and in the special case $f(x)=x$ it is known as star-discrepancy
$D_N^*.$ In this case a sequence $(x_n)$ satisfying \eqref{eq:1} is called
uniformly distributed modulo $1$ (uniformly distributed modulo $1$). For the basic facts in uniform
distribution theory we refer to the monographs of Kuipers and Niederreiter
\cite{kuipers_niederreiter1974:uniform_distribution_sequences}, Drmota and
Tichy~\cite{drmota_tichy1997:sequences_discrepancies_and} or Strauch and
Porubsk\'y \cite{strauch_porubsky2005:distribution_sequences}.

Since the beginning of the 20th century the Kronecker sequence $(n\alpha)_n$ is
known to be uniformly distributed modulo $1$ for every irrational number
$\alpha.$ Furthermore, discrepancy estimates have been established depending on
the diophantine approximation properties of the number $\alpha.$ For this
purpose we define the approximation type of a real number and (since we will use
it in the second part of this paper) more generally, for a $s-$tuple of real
numbers.
\begin{defi}
  Let $(\gamma_1,\ldots,\gamma_s)\in\RR^s$ be a vector of reals and $t>0$. Then
  we say that $(\gamma_1,\ldots,\gamma_s)$ is of finite type $t$ if there exists
  a constant $c>0$ depending on $\gamma_1,\ldots,\gamma_s$ and $\varepsilon>0$
  such that for all vectors $\left(n_1,\ldots,n_s\right)\in\ZZ^*$ we have
  \[
    \prod_{j=1}^s\left(\max\left(1,\abs{n_j}\right)\right)^{t+\varepsilon}
    \norm{\sum_{j=1}^s n_j\gamma_j}\geq c.
  \]
\end{defi}
Now it is well-known that the discrepancy can be estimated in terms of exponential sums, for instance applying the following standard tool in uniform distribution theory



\begin{lem}[Erd\H{o}s-Turán]\label{lem:erdos-turan}
  For any integers $N>0$, $H>0$, and any sequence $(x_n)_{n=1}^N$ of $N$ real
  numbers we have
  \[
    D_N^*\left((x_n)_{n=1}^N\right)
    \leq \frac{2}{H+1} + 2\sum_{h=1}^H\frac1h
    \abs{\frac{1}{N}\sum_{n=1}^N e(hx_n)},
  \]
  where $e(t)=e^{2\pi it}.$
\end{lem}

\begin{proof}
  A proof of this inequality is Theorem 2.5 in Chapter 2 of Kuipers and
  Niederreiter \cite{kuipers_niederreiter1974:uniform_distribution_sequences}.
  For the coefficients in the present statement see the result of Rivat and
  Tenenbaum \cite{rivat_tenenbaum2005:constantes_derd_os}.
\end{proof}
Applying this inequality to the Kronecker sequence yields for real numbers
$\alpha$ of finite type $t$ the discrepancy estimate
\[
  D_N^*(\alpha n)\ll N^{-1/t+\varepsilon};
\]
for details we refer to Theorem 3.2 of Chapter 2 in
\cite{kuipers_niederreiter1974:uniform_distribution_sequences}. More refined
discrepancy estimates can be established using the continued fraction expansion
of $\alpha.$ In \cite{sos1976:discrepancy_sequence_n} V. T. Sós obtained a
formula for $D _N^*(\alpha n),$ which was later made completely explicit by
Schoissengeier \cite{schoissengeier1984:discrepancy_n_alpha}. 

Various authors have studied equidistribution properties of subsequences of
$(\alpha n)_{n=1}^\infty.$ For instance $(\alpha n^k)_{n=1}^\infty$ and $(\alpha
p_n)_{n=1}^\infty$ are known to be uniformly distributed mod $1,$ where $k$ is a
positive integer and $p_n$ denotes the $n$-th prime number. The first result is
due to Weyl \cite{weyl1916:ueber_die_gleichverteilung}, who considered sequences
defined by polynomials and the second one was proved by Vinogradov
\cite{vinogradov1948:estimate_trigonometric_sums}.

In the 1990's H\aa land \cite{ha1993:uniform_distribution_generalized,
ha1994:uniform_distribution_generalized} studied equidistribution properties of
so-called generalized polynomials. Such sequences are given by terms defined by
iterating addition, multiplication and the floor function $\lfloor . \rfloor.$ A
typical result in \cite{ha1993:uniform_distribution_generalized} is the
following: For $\alpha,\beta\in\RR$ the sequence
\[
  x_n=\lfloor \alpha n \rfloor \beta n
\]
is uniformly distributed modulo $1$ if and only if either $\alpha^2\not\in\QQ$ and $\beta\not\in\QQ$ or
$\alpha^2\in\QQ$ and $\beta$ is rationally independent of $1$ and $\alpha$.

Furthermore, she obtained diophantine conditions which guarantee uniform
distribution of generalized polynomials of higher degree. Her proofs depend on
Weyl's criterion using exponential sums and the difference method of van der
Corput; however no discrepancy bounds are given. In the first part of the
present paper we investigate the asymptotic distribution of generalized
polynomials of degree $2$ which are not uniformly distributed. In particular,
for integer $\ell\geq1$ we will consider the sequence
$w_n^{(\ell)}=\ell\left\lfloor u_n\right\rfloor u_n$, where $u_n$ is a uniformly
distributed sequence. This is motivated by the special case $w_n=\lfloor\sqrt 2
n\rfloor 2\sqrt 2 n$, for which  Ruzsa has shown that this sequence has the
asymptotic distribution function
\[
  g(x)=1-\sqrt{1-x};
\]
see the comment in the paper of H\aa land
\cite{ha1993:uniform_distribution_generalized}. We will extend this result to
obtain the following.

\begin{prop}\label{prop:distribution_function_ell=1}
  Let $(u_n)_{n=1}^\infty$ be uniformly distributed modulo $1$ and suppose that
  $u_n^2$ is an integer (for every $n$). Then the sequence $(
  w^{(2)}_n)=(2\left\lfloor u_n\right\rfloor u_n)$ is asymptotically distributed
  modulo $1$ with distribution function $g(x)=1-\sqrt{1-x}$. Furthermore, if  $(
  \frac {1}{2} u_n)$ is uniformly distributed modulo $1$ and  $u_n^2$ is an integer, then the
  sequence $( w^{(1)}_n)=(\left\lfloor u_n\right\rfloor u_n)$ is asymptotically
  distributed modulo $1$ with distribution function
  \[
    f(x)=\begin{cases}
      \frac{1}{2}\left(1-\sqrt{1-2x}\right) &\left(0\le x<\frac{1}{2}\right),\\
      \frac{1}{2}\left(2-\sqrt{2-2x}\right) &\left(\frac{1}{2}\le x\le 1\right).
    \end{cases}
  \]
\end{prop}

In Section \ref{sec:distribution_functions} we give a proof and provide discrepancy estimates as well as an
extension to sequences $( w^{(\ell)}_n)=(\ell\left\lfloor u_n\right\rfloor
u_n)_{n=1}^\infty$ with $\ell \in \NN$. This proposition immediately applies to
the sequences $(n^k\sqrt d)$ and $(p_n \sqrt d)$, where $k, d \in \NN$ and $d$
is not a square. A further interesting example is the sequence $(\sqrt n)$ which
is uniformly distributed modulo $1$ by Fejer's theorem (see Theorem 2.5 of Chapter 1 in
\cite{kuipers_niederreiter1974:uniform_distribution_sequences}).

In Sections \ref{sec:type_I_sum_estimate} to
\ref{sec:discrepancy_second_iteration} we prove a discrepancy estimate for
sequences of the form 
\[
  u_n=\left(\beta\left\lfloor\alpha_2\left\lfloor\alpha_1
    f\left(p_n\right)\right\rfloor\right\rfloor\right),
\]
where $f$ is a monic polynomial of degree $d\geq2$ and $\left(p_n\right)_n$ is
the sequence of primes. We establish this estimate in several steps starting
with the Erd\H{o}s-Turán inequality (Lemma \ref{lem:erdos-turan}) rewriting the
estimate of the discrepancy into one of an exponential sum. Then Vaughan's
identity (\textit{cf.} Lemma~\ref{lem:vaughans_identity} below) allows us to
transform the sum over the primes into a linear combination of bilinear forms of
two different types. These two types are treated in Section
\ref{sec:type_I_sum_estimate} and Section \ref{sec:type_II_sum_estimate},
respectively.

The main idea is to unfold the floor functions. Therefore we need to estimate
the discrepancy of the sequences
\[
  \left(\beta\left\lfloor\alpha
    f\left(p_n\right)\right\rfloor\right)_n
  \quad\text{and}\quad
  \left(\beta
    f\left(p_n\right)\right)_n,
\]
which we do in Section \ref{sec:discrepancy_first_iteration} and Section
\ref{sec:discrepancy_finite_type_primes}, respectively. Finally in Section
\ref{sec:discrepancy_second_iteration} we will prove the following.

\begin{thm}\label{thm:discrepancy_estimate_second_iteration} Let $f\in\RR[X]$ be
  a monic polynomial of degree $d\geq2$ and let $\alpha_1$, $\alpha_2$ and
  $\beta$ be three reals. Suppose that
  $(\alpha_1\alpha_2\beta,\alpha_1\alpha_2,\alpha_1)$ is of finite type $t>0$.
  Then there exists $\eta>0$ depending only on $d$ and $t$ such that for $\varepsilon>0$ we have
  \[
    D_{P}\left((\beta f(p_n))_{n=1}^P\right)\ll
    P^{-\eta+\varepsilon}.
  \]
\end{thm}

\begin{rem}
  For given $d$ and $t$ one may explicitely calculate the exponent $\eta$
  following the proof.
\end{rem}

\section{Distribution functions}\label{sec:distribution_functions} In this
section we will prove Proposition \ref{prop:distribution_function_ell=1} and
some generalizations and refinements. Throughout this section we suppose that
$(u_n)$ is a uniformly distributed sequence such that $u_n^2\in\ZZ$ for all
$n\in\NN$. Then for integers $\ell\geq1$ we calculate the distribution function
of the sequence
\[
  w^{(\ell)}_n=\ell 
  \left\lfloor u_n\right\rfloor u_n.
\]

We start our considerations with the even cases $\ell=2k$ with $k\geq1$. For
$\ell=2$ we note that
\begin{align*}
  w^{(2)}_n
  &=\left\lfloor u_n\right\rfloor u_n
    + \left\lfloor u_n\right\rfloor u_n\\
  &=\left\lfloor u_n\right\rfloor
    \left(\left\lfloor u_n\right\rfloor + \left\{u_n\right\}\right)
    + \left(u_n-\left\{ u_n\right\}\right) u_n\\
  &=\left\lfloor u_n\right\rfloor^2
    + \left\lfloor u_n\right\rfloor\left\{u_n\right\}
    + u_n^2-\left\{ u_n\right\}u_n\\
  &=\left\lfloor u_n\right\rfloor^2
    + u_n^2-\left\{ u_n\right\}^2.
\end{align*}
Since $u_n^2\in\ZZ$ for all $n\in\NN$ we obtain modulo $1$ that
\[
  w_n^{(2)}\equiv-\left\{u_n\right\}^2
  \equiv 1-\left\{u_n\right\}^2
  \equiv \left\{w_n^{(2)}\right\}.
\]

Following the same lines we obtain for an arbitrary even $\ell=2k$ that
\[
  w^{(2k)}_n
  =k\left(2\left\lfloor u_n\right\rfloor u_n\right)
  \equiv k\left(1-\left\{u_n\right\}^2\right).
\]
Thus
\[
  \#\left\{ n\leq N\colon \left\{ w^{(2k)}_n\right\}<x\right\}
  =\sum_{j=1}^k \#\left\{ n\leq N\colon \sqrt{\frac{j-x}{k}}\leq
    \left\{u_n\right\}<\sqrt{\frac{j}{k}}\right\}.
\]
Since $(u_n)_n$ is uniformly distributed modulo $1$, we have
\[
  \frac{1}{N}\#\left\{ n\leq N\colon A\leq\{u_n\}<
    B\right\}\xrightarrow[N\to\infty]{} B-A.
\]
Therefore we obtain the distribution function
\[
  G(x)
  =\lim_{N\to\infty}
  \frac{1}{N}\#\left\{ n\leq N\colon \left\{ w^{(2k)}_n\right\}<x\right\}
  =\sum_{j=1}^k\left(\sqrt{\frac{j}{k}}-\sqrt{\frac{j-x}{k}}\right),
\]
which reduces in the case $k=1$ to the function $g(x)$ defined in Proposition
\ref{prop:distribution_function_ell=1}.

The following lemma immediately yields a discrepancy bound.

\begin{lem} Let $(x_n)_{n=1}^\infty$ be a uniformly distributed sequence in $[0,1[$ with discrepancy $D_N^*$ and let $\psi: [0,1]\to [0,1]$ be a strictly decreasing function with $\psi (0)=1$ and $\psi (1)=0$. 
Then the sequence $(\psi(x_n))_{n=1}^\infty$ is asymptotically distributed with distribution function $\phi= 1-\psi^{-1}$, and the discrepancy is bounded by $ D_N^{(\phi)}(\psi(x_n)) \le D_N^*$.

\end{lem}

\begin{proof}
  We have  
  \begin{align*}
    \frac{1}{N}\#\left\{ n\leq N\colon \psi(x_n)<x\right\}- \phi(x)
    &= \frac{1}{N}\#\left\{ n\leq N\colon x_n\in \left]\psi^{-1}(x), 1\right[ \right\}- \phi(x)\\
    &= 1-\psi^{-1}(x) -  \phi(x) + \Theta D_N^* =  \Theta D_N^*
  \end{align*}
  with $|\Theta | \le 1$. Taking the supremum over  $x\in \left]0,1\right]$ completes the proof.
\end{proof}

As an application we can take in Proposition
\ref{prop:distribution_function_ell=1} $(u_n)= (n\sqrt d)$, where $d$ is not a
square. Since  $\alpha=\sqrt d$ is an irrational with bounded continued fraction
expansion we obtain by Theorem 3.4 of Chapter 2 in
\cite{kuipers_niederreiter1974:uniform_distribution_sequences} that
\[
  D_N^{(g)}\left(2\left\lfloor n\sqrt{d}\right\rfloor n\sqrt{d}\right) \ll \frac{\log N}{N}.
\]

Now we turn our attention to the case of odd $\ell=2k-1$ with $k\geq1$, where
the situation is a little bit more involved. Again we start with the case $k=1$
which covers the assertion of
Proposition~\ref{prop:distribution_function_ell=1}. Here we have to additionally
assume that $\left( \tfrac {1}{2} u_n\right)$ is uniformly distributed modulo
$1$ and hence $(u_n)$ is uniformly distributed modulo $1$, too.

Let $u_{0,n}$ and $u_{1,n}$ be the sub-sequences of all $u_n$ such that $u_n^2$
is even or odd, respectively. Furthermore we denote by $w_{a,n}^{(\ell)}$ with
$a=0$ and $a=1$ the corresponding sub-sequences of $w_n^{(\ell)}$, \textit{i.e.}
\[
  w_{a,n}^{(\ell)}=\ell\left\lfloor u_{a,n}\right\rfloor u_{a,n}.
\]

We start with the case $\ell=1$ and follow the lines for even $\ell$ above. Then
we get that
\begin{align*}
  w_{a,n}^{(1)}&=\frac12\left(2\left\lfloor u_{a,n}\right\rfloor u_{a,n}\right)\\
  &=\frac12\left(\left\lfloor u_{a,n}\right\rfloor
    \left(\left\lfloor u_{a,n}\right\rfloor
    +\left\{ u_{a,n}\right\}\right)
    + \left(u_{a,n}-\left\{ u_{a,n}\right\}\right) u_{a,n}\right)\\
  &=\frac12\left(\left\lfloor u_{a,n}\right\rfloor^2
    -\left\{ u_{a,n}\right\}^2+a\right)
\end{align*}

Now we distinguish two cases according to the parity of $\lfloor
u_{a,n}\rfloor^2$. 
\begin{itemize}
  \item If $\lfloor u_{a,n}\rfloor\equiv a\pmod 2$, then
  \[
    w_{a,n}^{(1)}=-\tfrac12\left\{u_{a,n}\right\}^2
    =1-\tfrac12\left\{u_{a,n}\right\}^2
    =\left\{ w_{a,n}^{(1)}\right\} \in\left(\tfrac12,1\right]
  \]
  \item If $\lfloor u_{a,n}\rfloor\not\equiv a\pmod 2$, then
  \[
    w_{a,n}^{(1)}=\tfrac12-\tfrac12\left\{u_{a,n}\right\}^2
    =\left\{ w_{a,n}^{(1)}\right\} \in\left(0,\tfrac12\right]
  \]
\end{itemize}
Let $u_{a,m_j}$ be the subsequence corresponding to the first case and
$u_{a,n_j}$ the subsequence corresponding to the second one. Furthermore for
fixed $N\in \NN$ we set $M_J=\#\left\{m_j\leq N\right\}$ and
$N_J=\#\left\{n_j\leq N\right\}$. Since $( \frac {1}{2} u_n)$ is uniformly
distributed modulo $1$ we obtain
\[
  \lim_{N\to\infty}\frac{M_J}{N} = \lim_{N\to\infty}\frac{N_J}{N} = \frac {1}{2}.
\]
For the computation of the distribution function we will use the following
simple lemma.

\begin{lem}\label{lem:u.d._for_subintervals}
  Let $(x_n)$ be uniformly distributed in $[0,1)$ and $K\subset[0,1)$ a
  subinterval of length $|K|>0$. Furthermore, let $(x_{n_j})$ be the subsequence
  of $(x_n)$ with elements in $K$. Then $(x_{n_j})$ is uniformly distributed in
  $K$, \textit{i.e.}
  \[
    \lim_{N\to\infty}\frac{\#\left\{n\le N\colon  x_{n}\in I\right\}}
      {\#\left\{n\le N\colon  x_{n}\in K\right\}} = \frac{|I|}{|K|}
  \]
  for every subinterval $I$ of $K$.
\end{lem}

A proof of this theorem can be found in the book by Drmota and Tichy
\cite{drmota_tichy1997:sequences_discrepancies_and}. However, as the variant
there is very general, we shall provide a proof for this special case for
completeness.

\begin{proof}
  For a subinterval $I\subset[0,1[$ we denote by $\mathcal{A}_M(I)$ the number
  of elements $x_n$ with $n\leq M$ lying in $I$, \textit{i.e.}
  \[
    \mathcal{A}_M(I)=\#\left\{n\leq M\colon x_n\in I\right\}.
  \]
  Furthermore let $M=M(N)$ be the $N$th element of the sequence $(x_n)$ lying in
  $K$, \textit{i.e.} $\mathcal{A}_{M}(K)=N$. Since $(x_n)$ is uniformly distributed
  modulo $1$, we have
  \[
    \mathcal{A}_{M}(K)=N=M\cdot \abs{K}
      +o\left(M\right).
  \]
  Solving this for $M$ yields
  \[
    M=\frac{N}{\abs{K}}+o\left(M\right),
  \]
  where the implied constant now may depend on $K$. Plugging this into
  $\mathcal{A}_{M}(I)$ we obtain
  \[
    \mathcal{A}_{M}(I)
      =M\abs{I}
      +o\left(M\right)
      =N\frac{\abs{I}}{\abs{K}}
      +o\left(M\right).
  \]
  Thus
  \[
    \frac{\#\left\{n\le M\colon  x_{n}\in I\right\}}
      {\#\left\{n\le M\colon  x_{n}\in K\right\}}
    =\frac{\mathcal{A}_{M}(I)}{N}
    =\frac{\abs{I}}{\abs{K}}+
      o\left(\frac{M}{N}\right).
  \]
  Since $(x_n)$ is uniformly distributed we have $M/N\sim 1/\abs{K}$ and the
  lemma follows.
\end{proof}

Now we have all the tools at hand to finish the proof.

\begin{proof}[Proof of Proposition \ref{prop:distribution_function_ell=1}]
  For the computation of the distribution function we distinguish the two cases from above and obtain
\begin{itemize}
  \item If $0\leq x\leq \tfrac12$, then
  \begin{align*}
    &\frac1N\#\left\{n\leq N\colon \left\{ w_{a,n}^{(1)}\right\}<x\right\}\\
    &\quad=\frac{N_J}N\frac 1{N_J}\#\left\{n_j\leq N\colon \frac12-\frac12\left\{u_{a,n_j}\right\}^2<x\right\}\\
    &\quad=\frac{N_J}N\frac1{N_J}\#\left\{n_j\leq N\colon \left\{u_{a,n_j}\right\}\in \left(\sqrt{1-2x},1\right]\right\}.
  \end{align*}
  Since $(u_n)$ is  uniformly distributed modulo $1$, using Lemma
  \ref{lem:u.d._for_subintervals} we get 
  \[
    \frac1N\#\left\{n\leq N\colon \left\{ w_{a,n}^{(1)}\right\}<x\right\}
    \xrightarrow[N\to\infty]{} \frac12-\frac12\sqrt{1-2x}=:f_0(x).
  \]
  \item If $\tfrac12< x\leq \tfrac12$, then in the same vain as above we have
  \begin{align*}
    &\frac1N\#\left\{n\leq N\colon \frac12\leq \left\{ w_{a,n}^{(1)}\right\}<x\right\}\\
    &\quad=\frac{M_J}N\frac 1{M_J}\#\left\{m_j\leq N\colon 1-\frac12\left\{u_{a,m_j}\right\}^2<x\right\}\\
    &\quad=\frac{M_J}N\frac 1{M_J}\#\left\{m_j\leq N\colon \left\{u_{a,m_j}\right\}\in \left(\sqrt{2-2x},1\right]\right\}.
  \end{align*}
  Again by the uniform distribution of $(u_n)$ and Lemma
  \ref{lem:u.d._for_subintervals} we have
  \[
    \frac1N\#\left\{n\leq N\colon \frac12\leq \left\{ w_{a,n}^{(1)}\right\}<x\right\}
    \xrightarrow[N\to\infty]{} \frac12-\frac12\sqrt{2-2x}=:f_1(x).
  \]
\end{itemize}

For the complete distribution function of $(w_{n}^{(1))}$ we obtain
\begin{align*}
  f(x)&=\begin{cases}
    f_0(x) &\text{for }0\leq x<\frac12,\\
    f_0\left(\tfrac12\right)+f_1(x)&\text{for }\frac12\leq x\leq 1,
  \end{cases}\\
  &=\begin{cases}
    \tfrac12-\tfrac12\sqrt{1-2x} &\text{for }0\leq x<\frac12,\\
    1-\tfrac12\sqrt{2-2x}&\text{for }\frac12\leq x\leq 1,
  \end{cases}
\end{align*}
completing the proof.
\end{proof}

Finally we compute the distribution function $F(x)$ of the sequence $ w_{a,n}^{(\ell)}=\ell\left\lfloor u_{a,n}\right\rfloor u_{a,n}$ modulo $1$ for odd $\ell=2k-1$. We proceed as in the case $\ell=1$ and obtain
$$w_{a,n}^{(\ell)}=\frac k2\left(\left\lfloor u_{a,n}\right\rfloor^2
    -\left\{ u_{a,n}\right\}^2+a\right).$$
    
 Again we distinguish two cases whether the fractional part of $w_{a,n}^{(\ell)}$ lies in the interval $(0,\frac 12]$ or in $(\frac 12,1]$. The integer part $j$ of $w_{a,n}^{(\ell)}$ varies between $0$ and $k$. Thus we can proceed as in the case of even $\ell$ by taking the sum over $j$ which yields in the first case the distribution function
 $$F_0(x)=\frac 12 \sum_{j=0}^k\left(\sqrt{\frac{k-2j}{k}}-\sqrt{\frac{k-2j-2x}{k}}\right),$$
which corresponds to the function $f_0$ in the special case $\ell=1$. Similarly we obtain in the second case the distribution function
$$F_1(x)= \frac 12 \sum_{j=0}^k\left(\sqrt{\frac{k-2j}{k}}-\sqrt{\frac{k+1-2j-2x}{k}}\right) .$$
Then the sequence  $(w_{n}^{(2k-1)})$ has the asymptotic distribution function
\begin{align*}
  F(x)&=\begin{cases}
    F_0(x) &\text{for }0\leq x<\frac12,\\
    F_0\left(\tfrac12\right)+F_1(x)&\text{for }\frac12\leq x\leq 1.
  \end{cases}\\
\end{align*}

\section{Type I sum estimate}\label{sec:type_I_sum_estimate}

In the following sections we build up the tools we need for estimating the
discrepancy in Theorem~\ref{thm:discrepancy_estimate_second_iteration}.
Therefore we start in this and the next section by treating special weighted
exponential sums -- called type I and type II -- which occur after applying
Vaughan's identity (see Lemma \ref{lem:vaughans_identity} below).

We start with the type I sums, where the inner exponential sum has no weight. In
particular, we will obtain the following estimate.
\begin{lem}\label{lem:type-1-sum-estimate} Let $f\in\RR[X]$ be a monic
  polynomial of degree $d\geq2$. Suppose that
  $(\gamma_1,\ldots,\gamma_s)\in\RR^s$ is of finite type $t>0$. Then for any
  non-zero $(k_1,\ldots,k_s)\in\ZZ^s$ and sufficiently large $x$ we have
  \begin{multline*}
    T_1:=\sum_{M<m\leq 2M}a(m)\sum_{n\leq x/m}
      e\left(\left(k_1\gamma_1+\cdots+k_s\gamma_s\right) f(mn)\right)\\
    \ll M^{\frac{d-1}{R}}x^{1+\varepsilon}\left(\frac{M}{x}
      + \left(\frac{\abs{k_1\cdots k_s}^t}{x^{d-1}}\right)^{\frac{1}{st+1}}\right)^{\frac{1}{R}},
  \end{multline*}
  where $R=2^{d-1}$ and $a\colon \ZZ\to\RR$ is a function such that
  $\abs{a(m)}\ll_\varepsilon m^{\varepsilon}$ for $m\in\ZZ$.
\end{lem}

For the proof of this estimate we need several tools. The first one allows us to
transform the exponential sum into a sum of minima.
\begin{lem}\label{lem:sum_to_min} Let $\varepsilon>0$ and $f(x)=\alpha x^k+\beta
  x^{k-1}+\cdots$ be a polynomial with real coefficients of degree $k\geq2$.
  Then we have
  \[
    \abs{\sum_{n=1}^N e(f(n))}^K
    \ll N^{K-k+\varepsilon} \sum_{z=1}^{(k!) N^{k-1}}
    \min\left(N,\norm{z\alpha}^{-1}\right),
  \]
  where $K=2^{k-1}$ and the implied constant depends only on $k$ and $\varepsilon$.
\end{lem}

\begin{proof}
  This is a special case of Lemma 10C of Schmidt
  \cite{schmidt1977:small_fractional_parts}.
\end{proof}

The second tool links the estimation of the sum of minima with the discrepancy
of the sequence.
\begin{lem}\label{lem:min_to_discrepancy}
  Let $\theta\in\RR$ and let $L\in\NN^*$ be a positive integer. Then
  \[
    \sum_{\ell=1}^L\min\left(N,\norm{\ell\theta}^{-1}\right)
    \ll L\log N\left(1+ND_L\left(\left(\ell\theta\right)_{\ell=1}^L\right)\right).
  \]
\end{lem}

\begin{proof}
  This is Lemma 2.8 of Madritsch and Tichy \cite{madritsch_tichy2025:finite_pseudorandom_binary}.
\end{proof}

Finally, since our leading coefficient is of finite type we have the following
estimate for the discrepancy in Lemma \ref{lem:min_to_discrepancy}.
\begin{lem}\label{lem:discrepancy_of_n_alpha}
  Let $\theta\in\RR\setminus\QQ$ be an irrational number. Then for positive
  integers $L\geq1$ and $J\geq1$ the discrepancy of the sequence
  $\left(\ell\theta\right)_{\ell=1}^L$ satisfies
  \[
    D_L\left(\left(\ell\theta\right)_{\ell=1}^L\right)
    \leq C\left(\frac{1}{J}+\frac{1}{L}\sum_{j=1}^J\frac{1}{j\norm{j\theta}}\right),
  \]
  where $C$ is an absolute constant.
\end{lem}

\begin{proof}
  This is Lemma 3.2 of Kuipers and Niederreiter \cite{kuipers_niederreiter1974:uniform_distribution_sequences}.
\end{proof}

Since need a joined application of the last two tools for both, type I and type
II sums, we first combine them to get the following.
\begin{lem}\label{lem:sum_of_min_of_finite_type} Let $s\geq1$ be an integer and
  suppose that $(\gamma_1,\ldots,\gamma_s)\in\RR^s$ is of finite type $t>0$.
  Then for $(k_1,\ldots,k_s)\in\ZZ^*$ with $\abs{k_1\cdots k_s}^t\leq L$ we have
  \[
    \sum_{\ell=1}^L\min\left(N,\norm{(k_1\gamma_1+\cdots+k_s\gamma_s)\ell}^{-1}\right)
    \ll L\log N\left(1+N\left(\frac{\abs{k_1\cdots k_s}^t}{L}\right)^{\frac{1}{st+1}}\right).
  \]
\end{lem}

\begin{proof}
  In order to ease notation we write $\gamma=k_1\gamma_1+\cdots+k_s\gamma_s$ for
  short. Then by an application of Lemma \ref{lem:min_to_discrepancy} we obtain
  \[
    \sum_{\ell=1}^L\min\left(N,\norm{\gamma\ell}^{-1}\right)
    \ll L\log N\left(1+ND_L\left(\left(\gamma\ell\right)_{\ell=1}^L\right)\right).
  \]
  Since $(\gamma_1,\ldots,\gamma_s)$ is of finite type $t$, there exists a
  constant $c=c(\varepsilon,\gamma_1,\ldots,\gamma_s)>0$ such that for $j\geq1$
  \[
    \norm{\gamma j}=\norm{\left(k_1\gamma_1+\cdots+k_s\gamma_s\right)j}
    \geq \frac{c}{\abs{k_1\ldots k_s}^{t+\varepsilon}j^{st+\varepsilon}}.
  \]
  Now an application of Lemma \ref{lem:discrepancy_of_n_alpha} yields
  \begin{align*}
    D_L\left(\left(\gamma \ell\right)_{\ell=1}^L\right)
    &\ll \frac{1}{J}+\frac{1}{L}\sum_{j=1}^J j^{-1}j^{st+\varepsilon}\abs{k_1\cdots k_s}^{t+\varepsilon}\\
    &\ll \frac{1}{J}+\frac{1}{L}J^{st+\varepsilon}\abs{k_1\cdots k_s}^{t+\varepsilon}.
  \end{align*}
  Finally, putting
  \[
    J=\left(\frac{L}{\abs{k_1\cdots k_s}^t}\right)^{\frac{1}{st+1}}
  \]
  proves the lemma.
\end{proof}

Now we have collected all the tools for the proof of Lemma
\ref{lem:type-1-sum-estimate}.

\begin{proof}[Proof of Lemma \ref{lem:type-1-sum-estimate}]
  Again we write $\gamma=k_1\gamma_1+\cdots+k_s\gamma_s$ for short. Then an
  application of Hölder's inequality together with $\abs{a(m)}\ll
  m^{\varepsilon}$ yields
  \[
    \abs{T_1}^R\ll M^{R-1+\varepsilon}\sum_{M<m\leq 2M}
      \abs{\sum_{n\leq x/m}e\left(\gamma f(mn)\right)}^R.
  \]
  By an application of Lemma \ref{lem:sum_to_min} for the innermost sum we
  obtain
  \[
    \abs{T_1}^R\ll x^{R-d+\varepsilon}M^{d-1}\sum_{M<m\leq 2M}
      \sum_{z=1}^{(d!)(x/M)^{d-1}}
      \min\left(\frac{x}{M},\norm{\gamma m^dz}^{-1}\right).
  \]

  Now we want to combine the two sums over $m$ and $z$, respectively. Therefore
  we note that the number of solutions of $\ell=m^dz$ with $\ell\leq
  (d!)x^{d-1}M$ is $\ll \ell^{\varepsilon}\ll x^{\varepsilon}$ and simplify to get
  \[
    \abs{T_1}^R\ll x^{R-d+\varepsilon}M^{d-1}
      \sum_{\ell=1}^{(d!)x^{d-1}M}
      \min\left(\frac{x}{M},\norm{\gamma \ell}^{-1}\right).
  \]
  Finally applying Lemma \ref{lem:sum_of_min_of_finite_type} yields the desired bound.
\end{proof}

\section{Type II sum estimate}\label{sec:type_II_sum_estimate}

In this section we concentrate on the similar type II sums, where we also have a
weighted exponential sum. Therefore the proof is somehow more involved and we
will show the following.

\begin{lem}\label{lem:type-2-sum-estimate} Let $f\in\RR[X]$ be a monic
  polynomial of degree $d\geq2$. Suppose that
  $(\gamma_1,\ldots,\gamma_s)\in\RR^s$ is of finite type $t>0$. Then for any
  non-zero $(k_1,\ldots,k_s)\in\ZZ^s$ and sufficiently large $x$ we have
  \begin{multline*}
    T_2:=\sum_{M<m\leq 2M}a(m)\sum_{n\leq x/m} b(n) e\left(
      \left(k_1\gamma_1+\cdots+k_s\gamma_s\right) f(mn)\right)\\
    \ll x^{1+\varepsilon}\left(M^{-1}+\frac{M}{x}
      +\left(\frac{\abs{k_1\cdots k_s}^t}{x^{d-1}}\right)^{\frac{1}{st+1}}\right)^{\frac{1}{R^2}},
  \end{multline*}
  where $R=2^{d-1}$ and $a,b\colon \ZZ\to\RR$ are functions such that
  $\abs{a(m)}\ll_\varepsilon m^{\varepsilon}$ and $\abs{b(n)}\ll_{\varepsilon}
  n^{\varepsilon}$ for $m,n\in\ZZ$.
\end{lem}

Because of the weight in the exponential sum we need to unfold the Weyl
differencing, which is hidden in the proof of Lemma \ref{lem:sum_to_min}.
Therefore we need the following definition.

\begin{defi}
  Let $f\colon\ZZ\to\CC$ be a function and $r\in\ZZ$. Then we define the
  (forward) difference operator $\Delta_{r}$ on the function $f$ by
  \[
    \Delta_r(f)(x)=f(x+r)-f(x).
  \]
  Moreover, for $s\geq2$ and $r_1,\ldots,r_s\in\ZZ$ we recursively define the
  iterated difference operator $\Delta_{r_1,\ldots,r_s}$ by
  \[
    \Delta_{r_1,\ldots,r_s}=\Delta_{r_s}\circ \Delta_{r_1,\ldots,r_{s-1}}
    =\Delta_{r_s}\circ \Delta_{r_{s-1}}\circ\cdots\circ \Delta_{r_1}.
  \]
\end{defi}

The main idea is that for a polynomial $f$ we obtain that the degree of
$\Delta_r(f)$ is strictly less than the degree of $f$. Moreover if we
successively apply this differencing we obtain a linear function whose leading
coefficient is described by the following lemma.
\begin{lem}\label{lem:leading_coefficient_of_difference}
  For a polynomial $f(x)=\alpha x^k+\beta x^{k-1}+\cdots$ of degree $k\geq2$,
  \[
    \Delta_{y_1,\ldots,y_{k-1}} f(x)
    =y_1\cdots y_{k-1}\left(\frac{1}{2}k!\alpha\left(2x+y_1+\cdots+y_{k-1}\right)
      + (k-1)!\beta\right)
  \]
\end{lem}

\begin{proof}
  This is Lemma 10B of Schmidt \cite{schmidt1977:small_fractional_parts}.
\end{proof}

Now, together with the tools for the type I sum estimate from Section
\ref{sec:type_I_sum_estimate}, we have all we need.
\begin{proof}[Proof of Lemma \ref{lem:type-2-sum-estimate}]
  We write $\gamma=k_1\gamma_1+\cdots+k_s\gamma_s$ for short. Then squaring and
  inverting the order of summation we obtain
  \begin{align*}
    \abs{T_2}^2
    &=\abs{\sum_{n\leq x/M}b(n)\sum_{M<m\leq M'}a(m)e\left(\gamma f(mn)\right)}^2,
  \end{align*}
  where $M'=\min\{2M,x/m\}$. Applying the Cauchy-Schwarz
  inequality and squaring out yields
  \begin{align*}
    \abs{T_2}^2
    &\leq x^{1+\varepsilon}M^{-1}
      \left(\sum_{n\leq x/M}\sum_{M<m_1,m_2\leq M'}a(m_1)a(m_2)
      e\left(\gamma \left(f(m_1n)-f(m_2n)\right)\right)\right)\\
    &\leq x^{1+\varepsilon}M^{-1}
      \left(x^{1+\varepsilon}+\sum_{n\leq x/M}\sum_{\substack{M<m_1,m_2\leq M'\\m_1\neq m_2}}
      a(m_1)a(m_2) e\left(\gamma \left(f(m_1n)-f(m_2n)\right)\right)\right),
  \end{align*}
  where we have used that $\abs{b(n)}\ll n^\varepsilon$ for $n\in\ZZ^*$.

  Since we want to iterate this process we define for $\ell\geq1$ and $Y=\{y_1,\ldots,y_\ell\}$
  \[
    a(m,y_1,\ldots,y_\ell)=a(m)\prod_{j=1}^\ell 
    \prod_{\substack{\{y_{i_1},\ldots,y_{i_j}\}\subset Y\\ 1\leq i_1<\cdots <i_j\leq \ell}}
    a(m+y_{i_1}+\cdots+y_{i_j}).
  \]
  Then we may write $\abs{T_2}^2$ as
  \begin{gather}\label{eq:6}
    \abs{T_2}^2\ll x^{2+\varepsilon}M^{-1}+x^{1+\varepsilon}M^{-1}\abs{S_1},
  \end{gather}
  where 
  \[
    S_1 = \sum_{n\leq x/M}\sum_{y_1=1}^{M'-M}\sum_{m\in I(y_1)}
      a(m,y_1) e\left(\gamma \Delta_{y_1}f(mn)\right)
  \]
  and
  \[
    I(y_1)=\left\{M<m\leq M'\colon m+y_1\leq M'\right\}.
  \]

  Iterating this process we claim that for $\ell\geq1$ we have
  \begin{gather}\label{eq:T_2_iteration}
    \abs{T_2}^{2^\ell} \ll x^{2^{\ell}+\varepsilon} M^{-1}+x^{2^{\ell}-1+\varepsilon}M^{-\ell}\abs{S_{\ell}},
  \end{gather}
  where
  \[
    S_{\ell} = \sum_{n\leq x/M}\sum_{y_1=1}^{M'-M}\cdots\sum_{y_\ell=1}^{M'-M}
      \sum_{m\in I(y_1,\ldots,y_\ell)}
      a(m,y_1,\ldots,y_\ell) e\left(\gamma \Delta_{y_1,\ldots,y_\ell}f(mn)\right)
  \]
  and
  \[
    I(y_1,\ldots,y_\ell)=\left\{M<m\leq M'\colon m+y_1+\cdots+y_\ell\leq M'\right\}.
  \]

  In particular, the case $\ell=1$ is \eqref{eq:6} above and so we will only
  show the induction step $\ell\rightarrow \ell+1$. Squaring
  \eqref{eq:T_2_iteration} we obtain
  \begin{equation}\label{eq:T_2_iteration_step}
  \begin{split}
    \abs{T_2}^{2^{\ell+1}}
    &\ll x^{2^{\ell+1}+\varepsilon}M^{-1}+x^{2^{\ell+1}-2+\varepsilon}M^{-2\ell}\abs{S_\ell}^2,
  \end{split}
  \end{equation}
  where we have used that $\abs{S_\ell}\ll xM^\ell$.

  For $\abs{S_\ell}^2$ we use the Cauchy-Schwarz inequality and square out.
  Then we obtain analogously as above that
  \begin{align*}
    \abs{S_\ell}^2
    &=xM^{\ell-1}\abs{S_{\ell+1}}.
  \end{align*}
  Plugging this into \eqref{eq:T_2_iteration_step} proves the claim.

  Now an application of \eqref{eq:T_2_iteration} with $\ell=d-1$ yields
  \begin{gather}\label{eq:4}
    \abs{T_2}^{R}\ll x^{R+\varepsilon}M^{-1}+x^{R-1+\varepsilon}M^{-(d-1)}\abs{S_{d-1}},
  \end{gather}
  where we have set $R:=2^{d-1}$ for short. Again raising \eqref{eq:4} to the $R$-th power we get
  \begin{gather}\label{eq:5}
    \abs{T_2}^{R^2}
    \ll x^{R^2+\varepsilon}M^{-1}+x^{R^2-R+\varepsilon}M^{-(d-1)R}\abs{S_{d-1}}^R.
  \end{gather}

  We take a closer look at $\abs{S_{d-1}}^R$. First we exchangen the order of
  summation inside $S_{d-1}$. Thus we obtain
  \[
    S_{d-1}=\sum_{m}\sum_{y_1}\cdots\sum_{y_{d-1}} a(m,y_1,\ldots,y_{d-1})
      \sum_{n}e\left(\gamma \Delta_{y_1,\ldots,y_{d-1}}f(mn)\right).
  \]
  Since $a(m,y_1,\ldots,y_{d-1})\ll
  x^{\varepsilon}$ we obtain using Hölder's inequality
  \begin{align*}
    \abs{S_{d-1}}^R
    \ll M^{d(R-1)+\varepsilon}\sum_{m}\sum_{y_1}\cdots \sum_{y_{d-1}}
      \abs{\sum_{n} e\left(\gamma\Delta_{y_1,\ldots,y_{d-1}} f(mn)\right)}^R.
  \end{align*}

  Secondly we want to apply Lemma \ref{lem:sum_to_min}. Therefore we need the
  leading coefficient with respect to $n$. Since $f$ is of degree $d$, we get by
  Lemma \ref{lem:leading_coefficient_of_difference} that
  \begin{align*}
    \Delta_{y_1,\ldots,y_{d-1}} f(mn)
    &=y_1\cdots y_{d-1}\left(\frac12 d! n^d(2m+y_1+\cdots+y_{d-1})
      + (d-1)!\alpha_{d-1}y_1\cdots y_{d-1}n^{d-1}\right)\\
    &=:n^{d} h(m,y_1,\ldots,y_{d-1})+n^{d-1}(d-1)!\alpha_{d-1}y_1\cdots y_{d-1},
  \end{align*}
  where $\alpha_{d-1}$ is the coefficient of $X^{d-1}$ in $f$. Thus an application of Lemma \ref{lem:sum_to_min} yields for the innermost sum
  \[
    \abs{\sum_{n} e\left(\gamma\Delta_{y_1,\ldots,y_{d-1}} f(mn)\right)}^R
    \ll\left(\frac{x}{M}\right)^{R-d+\varepsilon}
    \sum_{z=1}^{(d!)(x/M)^{d-1}}
    \min\left(\frac{x}{M},\norm{\gamma h(m,y_1,\ldots,y_{d-1})z}^{-1}\right)
  \]
  and therefore
  \[
    \abs{S_{d-1}}^R
    \ll x^{R-d+\varepsilon} M^{(d-1)R}\sum_{m}\sum_{y_1}\cdots\sum_{y_{d-1}}\sum_{z}
    \min\left(\frac{x}{M},\norm{\gamma h(m,y_1,\ldots,y_{d-1})z}^{-1}\right).
  \]

  Now we combine all the summations by noting that for $\ell\leq (d!)^2x^{d-1}M$
  the number of solutions $(m,y_1,\ldots,y_{d-1},z)$ of
  \[
    \ell=h(m,y_1,\ldots,y_{d-1})z=y_1\cdots y_{d-1}\left(\frac12 d!(2m+y_1+\cdots+y_{d-1})\right)z
  \]
  is $\ll \ell^{\varepsilon}\ll x^{\varepsilon}$. Thus an application of Lemma
  \ref{lem:sum_of_min_of_finite_type} yields
  \begin{align*}
    \abs{S_{d-1}}^R
    \ll x^{R+\varepsilon} M^{(d-1)R}
    \left(\frac{M}{x}+\left(\frac{\abs{k_1\cdots k_s}^t}{x^{d-1}M}\right)^{\frac{1}{st+1}}\right).
  \end{align*}
  Together with \eqref{eq:5} this gives the desired bound.
\end{proof}

\section{Discrepancy estimate along the
primes}\label{sec:discrepancy_finite_type_primes}

In this section we start with a first step towards the proof of Theorem
\ref{thm:discrepancy_estimate_second_iteration} by estimating the discrepancy of
the sequence
\[
  \left(\beta f(p_n)\right)_{n=1}^P,
\]
where $p_n$ denotes the $n$-th prime

\begin{prop}\label{prop:basic_case}
  Let $f\in\RR[X]$ be a monic polynomial of degree $d\geq2$ and suppose
  that $\beta$ is of finite type $t>0$. Then for $\varepsilon>0$ we have
  \[
    D_{P}\left((\beta f(p_n))_{n=1}^P\right)\ll
    P^{-\frac{1}{2R^2}+\varepsilon}+P^{-\frac{d-1}{R^2(t+1)}+\varepsilon},
  \]
  where $R:=2^{d-1}$.
\end{prop}

As we have mentioned above its proof is based on the Erd\H{o}s-Turán inequality
rewriting the discrepancy estimate into one of an exponential sum. Since we use
this inequality also in the following sections we need to establish the
exponential sum estimate in a quite general form.

\begin{lem}\label{lem:exponential_sum_primes} Let $f\in\RR[X]$ be a monic
  polynomial of degree $d\geq2$. Suppose that
  $(\gamma_1,\ldots,\gamma_s)\in\RR^s$ is of finite type $t>0$. Then for any
  non-zero $(k_1,\ldots,k_s)\in\ZZ^s$ and sufficiently large $N$ we have
  \[
    \sums_{p\leq N}e\left((k_1\gamma_1+\cdots+k_s\gamma_s) f(p)\right)
    \ll 
    N^{1+\varepsilon}\left(N^{-\frac12}+\left(\frac{\abs{k_1\cdots k_s}^t}{N^{d-1}}\right)^{\frac{1}{st+1}}\right)^{\frac{1}{R^2}},
  \]
  where $\sums$ denotes a sum over primes.
\end{lem}

The rest of the section is devoted to the proof of this lemma. Therefore we
proceed in three steps, which are quite standard in this domain. In the first we
rewrite the sum over the primes into a weighted sum over the integers involving
the van Mangold function $\Lambda$, defined by

\[
  \Lambda(n)=\begin{cases}
    \log n &\text{if }n=p^k\text{ for some prime $p$ and integer $k\geq1$},\\
    0 &\text{otherwise},
  \end{cases}
\]
as weights.
\begin{lem}
  \label{lem:transition_primes_weighted_integers}
  Let $\psi$ be an arithmetic function such that $\abs{\psi(n)}\leq 1$ for all
  integers $n$. Then
  \[
    \abs{\sums_{p\leq N}\psi(N)}
    \ll \frac{1}{N}\max_{x\leq N}
    \abs{\sum_{n\leq x}\Lambda(n)\psi(n)}
    +\sqrt{N}.
  \]
\end{lem}

\begin{proof}
  This is Lemme 11 of Mauduit and Rivat \cite{mauduit_rivat2010:sur_un_probleme}.
\end{proof}

In the second step we use Vaughan's identity
(see \cite{vaughan1977:distribution_alpha_p}) in order to decompose the van
Mangold function $\Lambda$.
\begin{lem}[Vaughan's identity]
  \label{lem:vaughans_identity}
  Let $U$ and $V$ be positive integers. Then we may decompose the von Mangold function $\Lambda$ as
  \[
    \Lambda(n)=a_1(n)+a_2(n)+a_3(n)+a_4(n),
  \]
  where
  \begin{align*}
    a_1(n)&=\begin{cases}
      \Lambda(n)&\text{for }n\leq U,\\
      0&\text{otherwise},
    \end{cases}
    &
    a_3(n)&=-\sum_{\substack{dm=n\\m\leq UV}}\sum_{\substack{m=er\\ e\leq V, r\leq U}}\mu(e)\Lambda(r),\\
    a_2(n)&=\sum_{\substack{dm=n\\ m\leq V}}\log(d)\mu(m),
    &
    a_4(n)&=-\sum_{\substack{dm=n\\ m>V, d>U}}\Lambda(d)\left(\sum_{\substack{tr=m\\ r\leq V}}\mu(r)\right)
  \end{align*}
\end{lem}

\begin{proof}
  This is Vaughan's identity as it is stated in Chapter 24 of Davenport \cite{davenport1980:multiplicative_number_theory}.
\end{proof}

In the final step we note that the occurring sums are all type I or type II sums
and we use the corresponding estimates in Lemma \ref{lem:type-1-sum-estimate}
and Lemma \ref{lem:type-2-sum-estimate}, respectively.

\begin{proof}[Proof of Lemma \ref{lem:exponential_sum_primes}]
  First we abbreviate the notation by writing $\gamma =
  k_1\gamma_1+\cdots+k_s\gamma_s$. Then an application of Lemma
  \ref{lem:transition_primes_weighted_integers} yields
  \[
    \sums_{p\leq N} e\left(\gamma f(p)\right)
    \ll \frac{1}{\log N}\max_{x\leq N}\abs{\sum_{n\leq x}\Lambda(n) e\left(\gamma f(n)\right)}
      +\sqrt{N}.
  \]

  We concentrate on the inner sum on the right and use Vaughan's identity
  (Lemma \ref{lem:vaughans_identity}) to get that
  \[
    \sum_{n\leq x}\Lambda(n)e(\gamma f(n))
    \ll x^{1/3}\log N+S_1+S_2+S_3,
  \]
  where
  \begin{align*}
    S_1&=\sum_{m\leq x^{1/3}}\mu(m)\sum_{n\leq x/m}\log(n)e\left(\gamma f(mn)\right),\\
    S_2&=\sum_{m\leq x^{2/3}}\phi_1(m)\sum_{n\leq x/m} e\left(\gamma f(mn)\right),\\
    S_3&=\sum_{x^{1/3}<m\leq x^{2/3}}\phi_2(m)
      \sum_{x^{1/3}<n\leq x/m}\Lambda(n) e\left(\gamma f(mn)\right)
  \end{align*}  
  with
  \begin{align*}
    \phi_1(m)=\sum_{\substack{m=er\\ e,r\leq x^{1/3}}}\mu(e)\Lambda(r)\ll \log m
    \quad\text{and}\quad
    \phi_2(m)=\sum_{\substack{tr=m\\ r\leq x^{1/3}}}\mu(r)\ll \tau(m),
  \end{align*}
  where $\mu$ and $\tau$ are the Möbius and the divisor function, respectively.
  
  Starting with $S_1$ we partition the sum into dyadic sub-sums, \textit{ }
  \[
    S_1=\sum_{j=0}^{\left\lfloor\frac12\frac{\log x}{\log 2}\right\rfloor}
      S_{1j}
  \]
  with
  \[
    S_{1j}=\sum_{m=2^{j}+1}^{\min\left(2^{j+1},x^{1/3}\right)}\mu(m)
      \sum_{n\leq x/m}\log(n)e\left(\gamma f(mn)\right).
  \]
  Setting
  \[
    J:=\min\left(x^{1/2},\left(\frac{\abs{k_1\cdots k_s}^t}{x^{d-1}}\right)^{\frac{1}{st+1}}\right)
  \]
  we distinguish the two cases $2^{jR}\geq J$ and $2^{jR}<J$. In the first case
  we directly apply Lemma \ref{lem:type-2-sum-estimate}. In the second case,
  however, we first do partial summation and then apply Lemma
  \ref{lem:type-1-sum-estimate}. Putting both cases together we get the estimate
  \[
    S_{1}\ll x^{1+\varepsilon}\left(x^{-\frac12}+\left(\frac{\abs{k_1\cdots k_s}^t}{x^{d-1}}\right)^{\frac{1}{st+1}}\right)^{\frac{1}{R^2}}.
  \]

  For $S_2$ we first divide the sum into two parts $S_2=S_{21}+S_{22}$ with
  \begin{align*}
    S_{21}&=\sum_{m\leq x^{1/2}}\phi_1(m)\sum_{n\leq x/m} e(\gamma f(mn))\quad\quad\text{and}\\
    S_{22}&=\sum_{x^{1/3}<n\leq x^{1/2}}\sum_{m\leq x/n}\phi_1(m) e(\gamma f(mn)).
  \end{align*}
  We partition both sums into dyadic sub-sums and apply Lemma
  \ref{lem:type-1-sum-estimate} and Lemma \ref{lem:type-2-sum-estimate} for
  $S_{21}$ and $S_{22}$, respectively. Putting everything together we again get that
  \[
    S_{2}\ll x^{1+\varepsilon}\left(x^{-\frac12}+\left(\frac{\abs{k_1\cdots k_s}^t}{x^{d-1}}\right)^{\frac{1}{st+1}}\right)^{\frac{1}{R^2}}.
  \]

  Finally with $S_3$ we do the same split as with $S_2$. In particular, we write
  $S_3=S_{31}+S_{32}$ with
  \begin{align*}
    S_{31}&=\sum_{x^{1/3}<m\leq x^{1/2}}\phi_2(m)\sum_{x^{1/3<n\leq x/m}}\Lambda(n)e(\gamma f(mn))\quad\quad\text{and}\\
    S_{32}&=\sum_{x^{1/3}<n\leq x^{1/2}}\Lambda(n)\sum_{x^{1/3<m\leq x/n}}\phi_2(m)e(\gamma f(mn)).
  \end{align*}
  We partition $S_{31}$ into dyadic sub-sums, \textit{i.e.}
  \[
    S_{31}=\sum_{j=0}^{\left\lfloor\frac16\frac{\log x}{\log 2}\right\rfloor}S_{31j}
  \]
  with
  \[
    S_{31j}=\sum_{m=x^{1/3}2^{j}+1}^{\min\left(x^{1/3}2^{j+1},x^{1/2}\right)}
      \phi_2(m)\sum_{x^{1/3<n\leq x/m}}\Lambda(n)e(\gamma f(mn)).
  \]
  Then an application of Lemma \ref{lem:type-2-sum-estimate} yields
  \[
    S_{31}\ll x^{1+\varepsilon}\left(x^{-\frac12}+\left(\frac{\abs{k_1\cdots k_s}^t}{x^{d-1}}\right)^{\frac{1}{st+1}}\right)^{\frac{1}{R^2}}.
  \]
  By similar means we get that $S_{32}$ also fulfills this bound and the lemma
  is proved.
\end{proof}

\begin{proof}[Proof of Proposition \ref{prop:basic_case}]
  First we set $N:=p_P$ to be the $P$-th prime and note that by the prime number
  theorem $P(\log P)\ll N\ll P(\log P)$ (\textit{cf.} Theorem 4.5 of Apostol
  \cite{apostol1976:introduction_to_analytic}). Then an application of the
  Erd\H{o}s-Turán inequality (Lemma \ref{lem:erdos-turan}) yields
  \begin{equation}\label{eq:7}
    D_{P}\left((\beta f(p_n))_{n=1}^P\right)
    \ll \frac{1}{H}+\sum_{h=1}^H\frac{1}{h}
      \abs{\frac{1}{N}\sums_{p\leq N}e\left(h\beta f(p)\right)},
  \end{equation}
  where $H\geq1$ is an integer we will choose below.

  Now we apply Lemma \ref{lem:exponential_sum_primes} with $s=1$, $k_1=h$ and
  $\gamma_1=\beta$ and obtain
  \[
    \abs{\frac{1}{N}\sums_{p\leq N}e\left(h\beta f(p)\right)}
    \ll N^{\varepsilon}\left(N^{-\frac12}+\left(\frac{h^t}{N^{d-1}}\right)^{t+1}\right)^{\frac{1}{R^2}}.
  \]
  Plugging this into \eqref{eq:7} yields
  \begin{align*}
    D_{P}\left((\beta f(p_n))_{n=1}^P\right)
    &\ll \frac{1}{H}+N^{-\frac{1}{2R^2}+\varepsilon}
      + N^{-\frac{1}{R^2}\frac{d-1}{t+1}+\varepsilon}
      H^{\frac{1}{R^2}\frac{t}{t+1}},    
  \end{align*}
  where we have used Theorem 3.2 of Apostol
  \cite{apostol1976:introduction_to_analytic} for the sum over the $h$.

  Finally, choosing
  \[
    H=\left\lceil N^{\frac{d-1}{t+R^2(t+1)}}\right\rceil\geq1
  \]
  proves the proposition.
\end{proof}

\section{Discrepancy estimate for the first
iteration}\label{sec:discrepancy_first_iteration}

In this section we shift towards our second intermediate step by estimating the
discrepancy of a sequence of the form
\[
  x_n=\beta\left\lfloor \alpha f(p_n)\right\rfloor.
\]

Again the approach is clear and we begin with the Erd\H{o}s-Turán inequality
transforming the problem into an exponential sum estimate. In the sequel we need
to treat the floor function. To this end, for reals $x,\tau\in\RR$ we define
\begin{equation}\label{eq:F_definition}
  F(x,\tau)=e\left(\tau\left\{x \right\}\right).
\end{equation}

The main problem with this function is that it is not continuous at integer
$x$. Thus we smooth it by considering the 
$r$-fold convolution
\begin{equation}\label{eq:Gr_definition}
  G_r(x,\tau,\delta)=\frac{1}{(2\delta)^r}
  \left(1\!\!1_{[-\delta,\delta]} \star \cdots \star 1\!\!1_{[-\delta,\delta]}
  \star F(x,\tau)\right),
\end{equation}
where $r\geq1$ is an integer and $\delta>0$.

The following Lemma controls the error when switching from $F$ to $G_r$.
\begin{lem}\label{lem:transition_F_to_Gr}
  Let $r\geq1$ be an integer and $\delta>0$ be a real. For any sequence
  $\{u_n\}_{n\geq1}$ of real numbers, and any positive integer $N$, we have
  \[
    \sum_{n\leq N}\abs{F(u_n,\tau)-G_r(u_n,\tau,\delta)}
    \ll Nr\delta+Nr^2\delta\abs{\tau}
    +ND_N\left(\left(u_n\right)_{n=1}^N\right).
  \]
\end{lem}

\begin{proof}
  This is Lemma 10 of Hofer and Ramaré
  \cite{hofer_ramare2016:discrepancy_estimates_some}. 
\end{proof}

Since the function $G_r$ is smoothed out by the folding, we consider its
Fourier series:
\begin{equation}\label{eq:Gr_Fourier_series}
  G_r(x,\tau,\delta)=\sum_{k\in\ZZ} \widehat{G_r}(k,\tau,\delta) e\left(-kx\right).
\end{equation}
The following two lemmas provide us with the reason for considering $G_r$ instead
of $F$. The first one tells us that the Fourier coefficients are exponentially
decreasing.
\begin{lem}\label{lem:exponential_decay_Fourier_Gr}
  Let $K$ be a positive integer such that $\abs{\tau+k}\geq\frac{k}{2}$ for
  $k\in \ZZ$ with $\abs{k}>K$. Then
  \[
    \sum_{\abs{k}>K}\widehat{G_r}(k,\tau,\delta)
    \ll \left(\delta K\right)^{-r}.
  \]
\end{lem}

\begin{proof}
  This is Lemma 3 of Mukhopadhyay \textit{et al.} \cite{mukhopadhyay_ramare_viswanadham2018:discrepancy_estimates_generalized}.
\end{proof}

The second lemma provides us with an estimate for the $p$-th moment for $p>1$.
\begin{lem}\label{lem:pth_moment_of_Gr}
  Let $\tau\in\RR$ and $0<\delta<\min\left(\frac{1}{2\abs{\tau}},1\right)$. Then
  for $p>1$ we have
  \[
    \sum_{k\in\ZZ}\abs{\widehat{G_r}(k,\tau,\delta)}^p
    \ll_p1.
  \]
\end{lem}

\begin{proof}
  This is Lemma 4 of Mukhopadhyay \textit{et al.} \cite{mukhopadhyay_ramare_viswanadham2018:discrepancy_estimates_generalized}.
\end{proof}

Now we can successfully attack the main result of this section.
\begin{prop}\label{prop:discrepancy_estimate_first_iteration}
  Let $f\in\RR[X]$ be a monic polynomial of degree $d\geq2$ and suppose that
  $\alpha$ and $\beta$ are reals such that $(\alpha,\alpha\beta)$ is of finite
  type $t>0$. Then for $\varepsilon>0$ we have
  \[
    D_P\left(\left(\beta\left\lfloor\alpha
    f(p_n)\right\rfloor\right)_{n=1}^P\right)
    \ll P^{-\frac{1}{2R^2+4}+\varepsilon}
      + P^{-\frac{d-1}{R^2(2t+1)+7t+2}}
  \]
\end{prop}

\begin{proof}
  Using the Er\H{o}s-Turàn inequality (Lemma \ref{lem:erdos-turan}) we obtain
  \begin{equation}\label{eq:12}
    D_P\left(\left(\beta\left\lfloor\alpha
    f(p_n)\right\rfloor\right)_{n=1}^P\right)
    \ll \frac{1}{H}+\sum_{h=1}^H\frac{1}{h}
    \abs{\frac{1}{P}\sums_{p\leq N}
      e\left(h\beta\left\lfloor\alpha f(p)\right\rfloor\right)},
  \end{equation}
  where $H$ is some positive integer we will choose below and $\sums$ denotes a
  sum over primes.

  We concentrate on the innermost sum over the primes. Using the definition of
  $F$ in \eqref{eq:F_definition} we get
  \[
    \sums_{p\leq N}
      e\left(h\beta\left\lfloor\alpha f(p)\right\rfloor\right)
    =\sums_{p\leq N}
      e\left(h\beta\alpha f(p)\right)F\left(\alpha f(p),-h\beta\right).
  \]
  Now an application of Lemma \ref{lem:transition_F_to_Gr} yields
  \begin{equation}\label{eq:11}
    \sums_{p\leq N}
      e\left(h\beta\alpha f(p)\right)F\left(\alpha f(p),-h\beta\right)
    =\sums_{p\leq N}
      e\left(h\beta\alpha f(p)\right)G_r\left(\alpha f(p),-h\beta,\delta\right)+\mathcal{O}\left(S_1\right),
  \end{equation}
  where $r\geq1$ is an integer and $\delta>0$, which we will choose below and
  \begin{equation*}
    S_1=\sums_{p\leq N}\abs{F\left(\alpha f(p),-h\beta\right)
      -G_r\left(\alpha f(p),-h\beta,\delta\right)}.
  \end{equation*}

  First we consider the weighted exponential sum on the right of
  \eqref{eq:Gr_Fourier_series}. Using the Fourier series of $G_r$ we get
  \[
    \sums_{p\leq N}
      e\left(h\beta\alpha f(p)\right)G_r\left(\alpha f(p),-h\beta,\delta\right)
    =\sum_{k\in\ZZ}\widehat{G_r}\left(k,-h\beta,\delta\right)\sums_{p\leq N}
      e\left(\left(h\beta\alpha -k\alpha\right)f(p)\right).
  \]

  Now we recall that the Fourier coefficients decay exponentially (Lemma
  \ref{lem:exponential_decay_Fourier_Gr}) and obtain
  \begin{multline}\label{eq:8}
    \sum_{k\in\ZZ}\widehat{G_r}\left(k,-h\beta,\delta\right)\sums_{p\leq N}
        e\left(\left(h\beta\alpha -k\alpha\right)f(p)\right)\\  =    \sum_{\abs{k}\leq K}\widehat{G_r}\left(k,-h\beta,\delta\right)\sums_{p\leq N}
        e\left(\left(h\beta\alpha -k\alpha\right)f(p)\right)
  +\mathcal{O}\left(P\left(\delta K\right)^{-r}\right),
  \end{multline}
  where $K\geq0$ is some parameter we will choose below.

  By Hölder's inequality we have that
  \begin{multline*}
    \abs{\sum_{\abs{k}\leq K}\widehat{G_r}\left(k,-h\beta,\delta\right)\sums_{p\leq N}
        e\left(\left(h\beta\alpha -k\alpha\right)f(p)\right)}\\
    \leq \left(\sum_{\abs{k}\leq K}\abs{\widehat{G_r}\left(k,-h\beta,\delta\right)}^{\frac{R^2}{R^2-1}}\right)^{\frac{R^2-1}{R^2}}
        \left(\sum_{\abs{k}\leq K}\abs{\sums_{p\leq N}
        e\left(\left(h\beta\alpha -k\alpha\right)f(p)\right)}^{R^2}
        \right)^{\frac{1}{R^2}}.
  \end{multline*}
  Since $R^2/(R^2-1)>1$ we get by Lemma \ref{lem:pth_moment_of_Gr} that
  \begin{multline*}
    \abs{\sum_{\abs{k}\leq K}\widehat{G_r}\left(k,-h\beta,\delta\right)\sums_{p\leq N}
        e\left(\left(h\beta\alpha -k\alpha\right)f(p)\right)}
    \ll
        \left(\sum_{\abs{k}\leq K}\abs{\sums_{p\leq N}
        e\left(\left(h\beta\alpha -k\alpha\right)f(p)\right)}^{R^2}
        \right)^{\frac{1}{R^2}}.
  \end{multline*}

  The vector $(\alpha\beta,\alpha)$ is of finite type $t>0$ and an application
  of Lemma \ref{lem:exponential_sum_primes} yields
  \[
    \abs{\sums_{p\leq N}
        e\left(\left(h\beta\alpha -k\alpha\right)f(p)\right)}^{R^2}
    \ll N^{R^2+\varepsilon}
    \left(N^{-\frac12}+\left(\frac{\abs{hk}^{t}}{N^{d-1}}\right)^{\frac{1}{2t+1}}\right).
  \]
  Summing over $k$ we obtain
  \begin{equation}\label{eq:9}
    \sum_{\abs{k}\leq K}\abs{\sums_{p\leq N}
      e\left(\left(h\beta\alpha -k\alpha\right)f(p)\right)}^{R^2}
    \ll N^{R^2-\frac12+\varepsilon}K
      + N^{R^2-\frac{d-1}{2t+1}+\varepsilon}
      h^{\frac{t}{2t+1}} K^{\frac{3t+1}{2t+1}}.
  \end{equation}

  Now we set
  \[
    \delta^{-1}=h N^{\theta}\quad\text{and}\quad K=h^\rho N^\theta,
  \]
  where $\rho>0$ and $\theta>0$ are parameters we will choose below. Using this
  together with \eqref{eq:9} in \eqref{eq:8} yields
  \begin{multline}\label{eq:10}
    \sum_{k\in\ZZ}\widehat{G_r}\left(k,-h\beta,\delta\right)\sums_{p\leq N}
        e\left(\left(h\beta\alpha -k\alpha\right)f(p)\right)\\
    \ll N^{1+\varepsilon}\left(N^{-\frac{1-2\theta}{2R^2}}h^{\frac{\rho}{R^2}}
      + N^{-\frac{d-1-\theta(3t+1)}{R^2(2t+1)}}
      h^{\frac{\rho(3t+1)+t}{R^2(2t+1)}}
      + h^{r(1-\rho)}\right),
  \end{multline}

  We return to the error from switching from $F$ to $G_r$. Applying Lemma \ref{lem:transition_F_to_Gr} and Lemma \ref{prop:basic_case} we
  obtain
  \begin{align*}
    S_1&\ll Pr\delta+Pr^2\delta h+PD_{P}\left(\left(\alpha f(p_n)\right)_{n=1}^{P}\right)\\
    &\ll P\left(rh^{-1}N^{-\theta}+r^2N^{-\theta}
      + P^{-\frac{1}{2R^2}+\varepsilon}+P^{-\frac{d-1}{R^2(t+1)}+\varepsilon}\right).
  \end{align*}
  Plugging this and \eqref{eq:10} into \eqref{eq:11}, then into \eqref{eq:12}
  and summing over $h$ yields
  \begin{multline*}
    D_P\left(\left(\beta\left\lfloor\alpha
    f(p_n)\right\rfloor\right)_{n=1}^P\right)\\
    \ll H^{-1}
      + N^{-\frac{1-2\theta}{2R^2}+\varepsilon}H^{\frac{\rho}{R^2}}
      + N^{-\frac{d-1-\theta(3t+1)+\varepsilon}{R^2(2t+1)}}
        H^{\frac{\rho(3t+1)+t}{R^2(2t+1)}}
      + H^{r(1-\rho)}
      + r^2N^{-\theta}
      .
  \end{multline*}

  We still need to set the parameters $\rho$, $r$, $H$ and $\theta$. Let
  $\rho=1+\varepsilon_1$ with $\varepsilon_1=\varepsilon_1(d,t)>0$ sufficiently small and let
  $r$ be an integer such that $r>\frac1{\varepsilon_1}$. Thus $H^{r(1-\rho)}\ll
  H^{-1}$ and we obtain
  \[
    D_P\left(\left(\beta\left\lfloor\alpha
    f(p_n)\right\rfloor\right)_{n=1}^P\right)
    \ll P^{-\theta+\varepsilon}
      + P^{-\frac{1-4\theta}{2R^2}+\varepsilon}
      + P^{-\frac{d-1-\theta(7t+2)}{R^2(2t+1)}+\varepsilon}.
  \]
  Depending on wether the second or third term on the right is dominant, we set
  $\theta$ accordingly. Thus
  \[
    \theta=\min\left(\frac{1}{2R^2+4}
    ,\frac{d-1}{R^2(2t+1)+7t+2}\right),
  \]
  which proves the proposition.
\end{proof}

\section{Proof of Theorem \ref{thm:discrepancy_estimate_second_iteration}}\label{sec:discrepancy_second_iteration}
In the same vain as in the section above we start with the Erd\H{o}s-Turán
inequality. For some integer $H\geq1$ we get
\begin{equation}\label{eq:18}
  D_P\left(\left(\beta\left\lfloor\alpha_1\left\lfloor\alpha_2
  f(p_n)\right\rfloor\right\rfloor\right)_{n=1}^P\right)
  \ll\frac{1}{H}+\sum_{h=1}^H\frac1h\abs{\frac{1}{P}
    \sums_{p\leq N}e\left(h\beta\left\lfloor\alpha_1\left\lfloor\alpha_2
    f(p)\right\rfloor\right\rfloor\right)},
\end{equation}
where again $\sums$ denotes a sum over primes.

We concentrate on the exponential sum and obtain using Lemma
\ref{lem:transition_F_to_Gr} that
\begin{align*}
  &\abs{\sums_{p\leq N}e\left(h\beta\left\lfloor\alpha_1\left\lfloor\alpha_2
    f(p)\right\rfloor\right\rfloor\right)}\\
  &\quad=\abs{\sums_{p\leq N}e\left(h\beta\alpha_1\left\lfloor\alpha_2
    f(p)\right\rfloor\right)F\left(\alpha_1\left\lfloor \alpha_2f(p)\right\rfloor,-h\beta\right)}\\
  &\quad\leq\abs{\sums_{p\leq N}e\left(h\beta\alpha_1\left\lfloor\alpha_2
    f(p)\right\rfloor\right)G_{r_1}\left(\alpha_1\left\lfloor \alpha_2f(p)\right\rfloor,-h\beta,\delta_1\right)}+\mathcal{O}\left(E_1\right),
\end{align*}
where $r_1$ and $\delta_1$ are parameters we choose below and
\[
  E_1=\sums_{p\leq N}\abs{F\left(\alpha_1\left\lfloor \alpha_2f(p)\right\rfloor,-h\beta\right)-G_{r_1}\left(\alpha_1\left\lfloor \alpha_2f(p)\right\rfloor,-h\beta,\delta_1\right)}.
\]

Continuing with the exponential sum, we consider the Fourier series of $G_r$
and get
\begin{multline*}
  \sums_{p\leq N}e\left(h\beta\alpha_1\left\lfloor\alpha_2
    f(p)\right\rfloor\right)G_{r_1}\left(\alpha_1\left\lfloor \alpha_2f(p)\right\rfloor,-h\beta,\delta_1\right)\\
  =\sum_{k_1\in\ZZ}\widehat{G_{r_1}}\left(k_1,-h\beta,\delta_1\right)
    \sums_{p\leq N}e\left((h\beta-k_1)\alpha_1\left\lfloor\alpha_2
    f(p)\right\rfloor\right).
\end{multline*}
Using the exponential decay of the Fourier coefficients in Lemma
\ref{lem:exponential_decay_Fourier_Gr} we obtain
\begin{multline*}
  \abs{\sums_{p\leq N}e\left(h\beta\alpha_1\left\lfloor\alpha_2
    f(p)\right\rfloor\right)G_{r_1}\left(\alpha_1\left\lfloor \alpha_2f(p)\right\rfloor,-h\beta,\delta_1\right)}\\
  \leq\abs{\sum_{\abs{k_1}\leq K_1}\widehat{G_{r_1}}\left(k_1,-h\beta,\delta_1\right)
    \sums_{p\leq N}e\left((h\beta-k_1)\alpha_1\left\lfloor\alpha_2
    f(p)\right\rfloor\right)}
    +\mathcal{O}\left(\left(\delta_1 K_1\right)^{-r}\right),
\end{multline*}
where $K_1$ is another parameter we will choose below.

Since there is still a floor function present we iterate the two last steps. Again an
application of Lemma \ref{lem:transition_F_to_Gr} yields
\begin{multline*}
  \abs{\sums_{p\leq N}e\left((h\beta-k_1)\alpha_1\left\lfloor\alpha_2
    f(p)\right\rfloor\right)}\\
  =\abs{\sums_{p\leq N}e\left((h\beta-k_1)\alpha_1\alpha_2
    f(p)\right)G_{r_2}\left(\alpha_2f(p),-(h\beta-k_1)\alpha_1,\delta_2\right)}+\mathcal{O}\left(E_2\right),
\end{multline*}
where $r_2$ and $\delta_2$ are parameters and
\[
  E_2=\sums_{p\leq N}\abs{F\left(\alpha_2f(p),-(h\beta-k_1)\alpha_1\right)-G_{r_2}\left(\alpha_2f(p),-(h\beta-k_1)\alpha_1,\delta_2\right)}.
\]
Then we use the exponential decay of the Fourier coefficients in Lemma
\ref{lem:exponential_decay_Fourier_Gr} in the second step to obtain
\begin{multline*}
  \abs{\sums_{p\leq N}e\left((h\beta-k_1)\alpha_1\alpha_2
    f(p)\right)G_{r_2}\left(\alpha_2f(p),-(h\beta-k_1)\alpha_1,\delta_2\right)}\\
  =\abs{\sum_{\abs{k_2}\leq K_2}
    \widehat{G_{r_2}}\left(k_2,-(h\beta-k_1)\alpha_1,\delta_2\right)
    \sums_{p\leq N}
    e\left((h\beta\alpha_1\alpha_2-k_1\alpha_1\alpha_2-k_2\alpha_2)f(p)\right)}\\
  +\mathcal{O}\left(P\left(\delta_2 K_2\right)^{-r_2}\right),
\end{multline*}
where $K_2$ is some parameter we will choose below.

In total we have five parts:
\[
  \abs{\sums_{p\leq N}e\left(h\beta\left\lfloor\alpha_1\left\lfloor\alpha_2
    f(p)\right\rfloor\right\rfloor\right)}
  \ll S_1+S_2+S_3+S_4+S_5,
\]
where
\begin{align*}
  S_1 &= \sum_{\abs{k_1}\leq K_1}\widehat{G_{r_1}}\left(k_1,-h\beta,\delta_1\right)
    \sum_{\abs{k_2}\leq K_2}\widehat{G_{r_2}}\left(k_2,-h\beta\alpha_1+k_1\alpha_1,\delta_2\right)\\
    &\quad\quad\times\sums_{p\leq N}e\left(\left(h\beta\alpha_1\alpha_2-k_1\alpha_1\alpha_2-k_2\alpha_2\right)f(p)\right),\\
  S_2 &= \sum_{\abs{k_1}\leq K_1}\widehat{G_{r_1}}\left(k_1,-h\beta,\delta_1\right)
    \left(\delta_2 K_2\right)^{-r_2},\\
  S_3 &= \sum_{\abs{k_1}\leq K_1}\widehat{G_{r_1}}\left(k_1,-h\beta,\delta_1\right)\\
    &\quad\quad\times
    \sums_{p\leq N}\abs{F\left(\alpha_2f(p),-h\beta\alpha_1+k_1\alpha_1\right)-G_{r_2}\left(\alpha_2f(p),-h\beta\alpha_1+k_1\alpha_1,\delta_2\right)},\\
  S_4 &= \left(\delta_1K_1\right)^{-r_1},\\
  S_5 &= \sums_{p\leq N}\abs{F\left(\alpha_1\left\lfloor \alpha_2f(p)\right\rfloor,-h\beta\right)-G_{r_1}\left(\alpha_1\left\lfloor \alpha_2f(p)\right\rfloor,-h\beta,\delta_1\right)},
\end{align*}
and $K_1$, $K_2$, $\delta_1$, $\delta_2$, $r_1$ and $r_2$ are parameters we will
choose at the end.

We start our considerations with $S_1$. Applying Hölder's inequality and Lemma
\ref{lem:pth_moment_of_Gr} yields
\[
  \abs{S_1}\ll
  \left(\sum_{k_1}\sum_{k_2}\abs{\sums_{p}
    e\left(\left(h\beta\alpha_1\alpha_2-k_1\alpha_1\alpha_2-k_2\alpha_2\right)f(p)\right)}^{R^2}\right)^{\frac{1}{R^2}}.
\]
By Lemma \ref{lem:exponential_sum_primes} and summing over $k_2$ we obtain
\begin{equation}\label{eq:13}
  \abs{S_1}^{R^2}\ll
  \sum_{k_1}
    \left(N^{R^2-\frac12+\varepsilon}K_2+N^{R^2-\frac{d-1}{3t+1}+\varepsilon}
    \abs{hk_1}^{\frac{t}{3t+1}}K_2^{\frac{4t+1}{3t+1}}\right).
\end{equation}
Similar to above we put
\[
  \delta_2^{-1}:=\abs{hk_1} N^{\theta_2}
  \quad\text{and}\quad
  K_2:=\abs{hk_1}^{\rho_2} N^{\theta_2},
\]
where $\rho_2$ and $\theta_2\geq0$ are two parameters. The first one we choose
such that $\rho_2=1+\varepsilon_2$ with $\varepsilon_2=\varepsilon_2(d,t)$ and
$r_2>\frac{1}{\varepsilon_2}$. For the second one, $\theta_2$,  we need to
consider $S_3$. In particular, by Hölder's inequality together with Lemma
\ref{lem:pth_moment_of_Gr} we get that
\[
  \abs{S_3}^{R^2}
  \ll \left(\sums_{p\leq N}\abs{F\left(\alpha_2f(p),-h\beta\alpha_1+k_1\alpha_1\right)-G_{r_2}\left(\alpha_2f(p),-h\beta\alpha_1+k_1\alpha_1,\delta_2\right)}\right)^{R^2}.
\]
Now we apply Lemma \ref{lem:transition_F_to_Gr} and Proposition
\ref{prop:basic_case} to get
\begin{multline}\label{eq:14}
  \sums_{p\leq N}\abs{F\left(\alpha_2f(p),-h\beta\alpha_1+k_1\alpha_1\right)-G_{r_2}\left(\alpha_2f(p),-h\beta\alpha_1+k_1\alpha_1,\delta_2\right)}\\
  \ll Pr\delta_2+Pr^2\delta_2\abs{hk_1}+PD_P\left(\left(\alpha_2f(p_n)\right)_{n=1}^P\right)\\
  \ll P\left(r\abs{hk_1}N^{-\theta_2}+r^2N^{-\theta_2}+P^{-\frac{1}{2R^2}+\varepsilon}+P^{-\frac{d-1}{R^2(t+1)}+\varepsilon}\right).
\end{multline}
We note that the last two terms in \eqref{eq:14} are dominated by the two terms
in \eqref{eq:13}, respectively. Therefore it suffices to choose $\theta_2$
according to the terms in \eqref{eq:13}. Considering the two cases of having
\[
  \frac{d-1}{3t+1}<\frac12
\]
or not we obtain two different choices for $\theta_2$, which we combine as
follows
\[
  \theta_2=\min\left(\frac{d-1}{R^2(3t+1)+4t+1},\frac{1}{2R^2+2}\right).
\]
Thus we get
\begin{equation}\label{eq:15}
  \abs{S_3}^{R^2}\ll \abs{S_1}^{R^2}
  \ll \sum_{k_1}
    N^{R^2-R^2\theta_2+\varepsilon}\abs{hk_1}^{\frac{t}{3t+1}}
  \ll N^{R^2-R^2\theta_2+\varepsilon}h^{\frac{t}{3t+1}}K_1^{\frac{4t+1}{3t+1}}.
\end{equation}

Similar to above we set
\[
  \delta_1^{-1}:=hN^{\theta_1}
  \quad\text{and}\quad
  K_1:=h^{\rho_1}N^{\theta_1},
\]
where $\rho_1$ and $\theta_1$ are two parameters. Again we choose
$\rho_1=1+\varepsilon_1$ with $\varepsilon_1=\varepsilon_1(d,t)$ such
that $r_1>\frac{1}{\varepsilon_1}$. For $\theta_1$ we need to consider $S_5$.
An application of Lemma \ref{lem:transition_F_to_Gr} and Proposition \ref{prop:discrepancy_estimate_first_iteration} yields
\begin{align*}
  S_5
  &\ll Pr_1\delta_1+Pr_1^2\delta_1\abs{h}+PD_P\left(\left(\alpha_1\left\lfloor\alpha f(p_n)\right\rfloor\right)_{n=1}^P\right)\\
  &\ll P\left(h^{-1}N^{-\theta_1}+N^{-\theta_1}+P^{-\frac{1}{2R^2+4}+\varepsilon}+P^{-\frac{d-1}{R^2(2t+1)+7t+1}+\varepsilon}\right).
\end{align*}
Again the two last terms are dominated by \eqref{eq:16}. Therefore it suffices
to consider the two cases of $\theta_2$ and we get that
\[
  \theta_1=\theta_2\frac{R^2(3t+1)}{R^2(3t+1)+4t+1}.
\]
Thus
\begin{equation}\label{eq:17}
  \abs{S_5}\ll \abs{S_1}
  \ll N^{1-\theta_1+\varepsilon}h^{\frac{5t+1}{R^2(3t+1)}}.
\end{equation}

Now we are left with estimating $S_2$ and $S_4$. On the one hand for $S_2$ we
again use Hölder's inequality and Lemma \ref{lem:pth_moment_of_Gr} to get
\begin{multline}\label{eq:16}
  \abs{S_2}^{R^2}
  \ll \left(\sum_{k_1}\abs{\widehat{G_{r_1}}\left(k_1,-h\beta,\delta_1\right)}^{\frac{R^2}{R^2-1}}\right)^{R^2-1}
    \left(\sum_{k_1}\left(\delta_2 K_2\right)^{-R^2r_2}\right)\\
  \ll \sum_{k_1}\abs{hk_1}^{-R^2}
  \ll h^{-R^2},
\end{multline}
where we have used our choices for $\delta_2$ and $K_2$ as well as that
$R^2\geq2$. On the other hand for $S_4$ we obtain
\[
  \abs{S_4}\ll h^{r_1(1-\rho_1)}\ll h^{-1}.
\]

Finally we plug all our estimates in \eqref{eq:14}, \eqref{eq:15}, \eqref{eq:17}
and \eqref{eq:16} into \eqref{eq:18} and sum over $h$ with $H=N^{-\theta_1}$
proving the theorem.

\section*{Acknowledgement}

Both authors acknowledge support from the bilateral project of the Agence
nationale de la recherche (ANR-23-CE40-0024) and the Austrian science fund (FWF,
I 6750).



\end{document}